\documentclass[reqno]{amsart}
\usepackage[utf8]{inputenc}
\usepackage{amsmath}
\usepackage{amsfonts}
\usepackage{amssymb}
\usepackage{amsthm}
\usepackage{mathrsfs}
\usepackage{enumitem}
\usepackage[a4paper, margin=1.5in]{geometry}

\newtheorem{theorem}{Theorem}
\newtheorem{lemma}{Lemma}
\newtheorem{corollary}{Corollary}
\newtheorem{proposition}{Proposition}
\theoremstyle{definition}
\newtheorem{definition}{Definition}
\newtheorem{example}{Example}
\newtheorem{question}{Question}
\newtheorem{remark}{Remark}

\newtheorem*{declaration}{Declaration of AI use}

\newcommand{\dnint}[1]{\left\|{#1}\right\|}
\newcommand{\floor}[1]{\left\lfloor{#1}\right\rfloor}
\renewcommand{\mod}{\;\mathrm{mod}\;}

\title[On a Notion of Density]{On a Notion of Density Related to Restricted Diophantine Approximation}
\author{Chance Sanford}
\address{Madison, Wisconsin, USA}
\email{sanfordchance@gmail.com}

\date{January 2026}

\begin{document}

\begin{abstract}
    We introduce notion of density for a subset of the natural numbers called the Diophantine density.  First, we demonstrate that the Diophantine density is an abstract density in the sense of Di Nasso and Jin.  Then we establish upper and lower bounds for the Diophantine density of a set and its complement respectively.  We also exhibit a construction which allows us to find a set whose density can be any value between zero and one.  Finally, we prove a restricted-denominator analog of Dirichlet's approximation theorem for sets of non-zero Diophantine density.
\end{abstract}

\maketitle
\tableofcontents
\section{Introduction}
One of the foundational results in theory Diophantine approximation is Dirichlet's approximation theorem.
\begin{theorem}[Dirichlet]\label{thm:dirichlet}
For any real numbers $\xi$ and $N \geq 1$, 
\begin{equation*}
    \min_{1\leq n\leq N}\dnint{n\xi} \leq (\lfloor{N}\rfloor + 1)^{-1} < N^{-1}.
\end{equation*}
\end{theorem}

 It is natural to wonder how the quality of approximation in Dirichlet's theorem would be affected if the approximating integers are restricted to some subset of the natural numbers. This question has received considerable attention over the years, with much of the research focused on restrictions to arithmetically significant subsets of the natural numbers. The arithmetic structure of these sets often allows for techniques from analytic number theory to be applied.  The following list gives a sample of the variety of sets which have been considered in the past.

\begin{itemize}
    \item Prime numbers (Matom{\"a}ki \cite{matomaki2009}).
    \item Square-free numbers (Harman \cite{harman1984diophantine}).
    \item Perfect squares (Zaharescu \cite{zaharescu1995small}).
    \item Numbers formed by the digits $0$ and $1$ in a given base (Iyer \cite{iyer2025rational}).
    \item Values of a quadratic form (Baier and Rahaman \cite{baier2025diophantine}).
\end{itemize}

To give a specific example, we present Zaharescu's main theorem from \cite{zaharescu1995small}. 
\begin{theorem}[Zaharescu]\label{thm:zaharescu}
    For any real number $\xi$, positive integer $N$, and $\varepsilon > 0$,
    \begin{equation*}
        \min_{1\leq n\leq N}\dnint{n^{2}\xi} \ll N^{-4/7+\varepsilon}.
    \end{equation*}
\end{theorem}

Results such as Zaharescu's theorem suggest studying subsets of the natural numbers according to the quality of Diophantine approximation they permit. In this paper, we take up this task by introducing the \textit{Diophantine density} of a set.  We show that dense sets in the Diophantine sense, also produce good Diophantine approximations.  Moreover, we hope to demonstrate that aside from its connection to rational approximation, the Diophantine density possesses a number of other interesting properties.

In the next section, we formally define Diophantine density and show that it possesses the properties of an \textit{abstract density}. Section 3 provides two theorems giving lower and upper bounds for the Diophantine density of a set. In Section 4, we construct a set with Diophantine density equal to any prescribed value in $(0,1)$. Section 5 establishes a connection back to Diophantine approximation by proving an analogue of Dirichlet's theorem in which the denominators of the approximating rationals are restricted to a set of positive Diophantine density.

Section 6 concerns sets arising from Piatetski-Shapiro sequences. Using the results of the preceding sections together with Zaharescu's theorem, we establish explicit lower and upper bounds for the Diophantine density of the set of perfect squares. Finally, Section 7 presents several open problems.

\section{Diophantine density}\label{sect:diophantine_density}
The following definition gives us our principal object of interest for the remainder of paper.
\begin{definition}\label{def:diophantine_density}
    The \textit{Diophantine density} $\delta(\mathcal{A})$ of a non-empty set $\mathcal{A} \subseteq \mathbb{N}$ is defined as the supremum of all real numbers $\delta$, such that there exists a $C(\mathcal{A},\delta) \in \mathbb{R}$ and a $q_{0}(\mathcal{A},\delta) \in \mathbb{N}$, for which
    \begin{equation*}
    \min_{v\in\mathcal{A},v < q}\dnint{\frac{v p}{q}} \leq C(\mathcal{A},\delta) q^{-\delta},
\end{equation*}
for every coprime pair $(p,q) \in \mathbb{Z} \times \mathbb{N}$ with $q \geq q_{0}(\mathcal{A},\delta)$.  For the empty set, we define $\delta(\varnothing) := 0$.
\end{definition}

The purpose of this section is to establish some of the basic properties of the Diophantine density and in particular, show that it possesses the features that one would expect a density to possess.  The question of what properties a ``generic" or ``abstract" density should satisfy has been raised by a number of people over the years.  The definition given below was proposed by Di Nasso and Jin \cite{dinasso2017abstract}.

\begin{definition}
    An \textit{abstract density} on $\mathbb{N}$ is a function $d : 2^{\mathbb{N}} \to [0,1]$ that satisfies the following properties:
    \begin{enumerate}[itemsep=0.25em, label=(\roman*)]
        \item $d(\mathbb{N}) = 1$.\label{item:full_density}
        \item $d(\mathcal{A}) = 0$ for every finite $\mathcal{A} \subset \mathbb{N}$.
        \item \textit{Monotonicity:} If $\mathcal{A} \subseteq \mathcal{B}$ then $d(\mathcal{A}) \leq d(\mathcal{B})$. \label{item:monotonicity}
    \end{enumerate}
\end{definition}

\begin{theorem}\label{thm:dd_is_abstract}
    The Diophantine density is an abstract density.
\end{theorem}

Our proof of Theorem \ref{thm:dd_is_abstract} will be spread out over the upcoming four propositions.  Each of which will establish one of Properties \ref{item:full_density} - \ref{item:monotonicity} as well as the fact that the codomain of the Diophantine density is $[0,1]$.  We begin by showing that the Diophantine density satisfies Property \ref{item:full_density}.

\begin{proposition}\label{prop:n_full_density}
    The Diophantine density of $\mathbb{N}$ is given by $\delta(\mathbb{N}) = 1$.
\end{proposition}
\begin{proof}
    For every coprime pair $(p,q) \in \mathbb{Z} \times \mathbb{N}$ we can set $v = (p^{-1} \mod q)$, since $p$ is invertible modulo $q$.  In doing so, we see that $v < q$ and
    \begin{equation*}
        \dnint{\frac{v p}{q}} = \frac{1}{q}.
    \end{equation*}
    This shows immediately that $\delta(\mathcal{\mathbb{N}}) \geq 1$.
    
    To see that this is the largest possible value for $\delta(\mathcal{\mathbb{N}})$, let us assume that some larger value is achievable.  Then there would be a $\delta > 1$ and $C> 0$ such that $q^{-1} \leq Cq^{-\delta}$.  But this implies that $q^{\delta-1} \leq C$ for all sufficiently large $q$, which is impossible since $\delta - 1 > 0$.  Therefore, we see that $\delta(\mathbb{N}) = 1$.
\end{proof}

\begin{proposition}\label{prop:zero_finite_sets}
    If $\mathcal{A} \subseteq  \mathbb{N}$ is finite, then $\delta(\mathcal{A})  = 0$.
\end{proposition}
\begin{proof}
    Aiming for a contradiction, let us assume that $\delta(\mathcal{A})  > 0$ and consider an infinite sequence of reduced rationals $(p_{n}/q_{n})_{n\geq 0}$, such that  $p_{n}/q_{n} \to \xi$ as $n \to \infty$.  By definition, for sufficiently large $n$, there are real numbers $\delta$ and $C(\mathcal{A},\delta)$ both greater than zero, such that
    \begin{equation*}
    \min_{v\in\mathcal{A},v < q_{n}}\dnint{\frac{v p_{n}}{q_{n}}} \leq C(\mathcal{A},\delta) q_{n}^{-\delta}.
\end{equation*}

Since $\mathcal{A}$ is finite and the sequence of rationals is infinite, by Dirichlet's pigeonhole principle, there must be a single $v \in \mathcal{A}$ that satisfies the above inequality for infinitely many $n$.  

Let us denote the subsequence of $n$'s for which this is true by $(n_{k})_{k\geq 0}$. By taking the limit as $k\to\infty$, we find that
\begin{equation*}
    \dnint{v \xi} = \lim_{k\to\infty}\dnint{\frac{v p_{n_{k}}}{q_{n_{k}}}} \leq C(\mathcal{A},\delta) \lim_{k\to\infty} q_{n_k}^{-\delta} = 0.
\end{equation*}
But this implies that $\xi$ is rational, which is a contradiction.  As a result, we conclude that $\delta(\mathcal{A})  = 0$.
\end{proof}

\begin{proposition}\label{prop:monotonicity}
    If $\mathcal{A} \subseteq \mathcal{B} \subseteq \mathbb{N}$, then $\delta(\mathcal{A}) \leq \delta(\mathcal{B})$.
\end{proposition}
\begin{proof}
    Since $\mathcal{A} \subseteq \mathcal{B}$ we see that for sufficiently large $q$,
    \begin{equation*}
        \min_{v\in\mathcal{B},v < q}\dnint{\frac{v p}{q}} \leq \min_{v\in\mathcal{A},v < q}\dnint{\frac{v p}{q}}
    \end{equation*}
    Thus, any admissible $\delta$ for the set $\mathcal{A}$ is also admissible for the set $\mathcal{B}$.  This shows that $\delta(\mathcal{B}) \geq \delta(\mathcal{A})$.
\end{proof}

\begin{proposition}\label{prop:delta_bounds}
    Let $\mathcal{A} \subseteq \mathbb{N}$.  The Diophantine density $\delta(\mathcal{A})$ is bounded by $0 \leq \delta(\mathcal{A}) \leq 1$.
\end{proposition}
\begin{proof}
    First, we note that $\dnint{x} \leq 1/2$ for all $x \in \mathbb{R}$.  Thus, for any pair $(p,q) \in \mathbb{Z} \times \mathbb{N}$ we have
    \begin{equation*}
        \min_{v\in\mathcal{A},v < q}\dnint{\frac{v p}{q}}  \leq \frac{1}{2}  < \frac{1}{q^{0}}.
    \end{equation*}
    As a result, $\delta(\mathcal{A}) \geq 0$.

    On the other hand, by Propositions \ref{prop:n_full_density} and \ref{prop:monotonicity} we know that for any subset $\mathcal{A} \subseteq \mathbb{N}$, $\delta(\mathcal{A}) \leq \delta(\mathbb{N}) = 1$.  Combining the lower and upper bounds completes the proof.
\end{proof}

With that we have shown that the Diophantine density is an abstract density and our proof of Theorem \ref{thm:dd_is_abstract} is complete.

In their paper on abstract densities, Di Nasso and Jin \cite{dinasso2017abstract} present four additional properties which an abstract density may possess.
\begingroup
\renewcommand{\thedefinition}{2a}
\begin{definition}  An abstract density may possess the following properties:

    \begin{enumerate}[itemsep=0.25em, start=4, label=(\roman*)]
    \item \textit{Subadditivity:} $d(\mathcal{A} \cup \mathcal{B}) \leq d(\mathcal{A}) + d(\mathcal{B})$. \label{item:subadditivity}
    \item \textit{Translation invariance:} $d(\mathcal{A} + k) = d(\mathcal{A} - k) = d(\mathcal{A})$ for every $\mathcal{A}\subseteq \mathbb{N}$ and $k \in \mathbb{N}$.\label{item:translation}
    \item \textit{Richness:} For every $x \in [0,1]$, there is an $\mathcal{A} \subseteq \mathbb{N}$ such that $d(\mathcal{A})= x$.
    \item \textit{Atomless-ness:} For every $\mathcal{A} \subseteq \mathbb{N}$ such that $d(\mathcal{A}) > 0$ there is a $\mathcal{B} \subseteq \mathcal{A}$ such that $0 < d(\mathcal{B}) < d(\mathcal{A})$.
\end{enumerate}
Here the right and left translates of $\mathcal{A}$ are defined by 
\begin{equation*}
    \mathcal{A} + k := \{a + k\, : \, a \in \mathcal{A}\} \quad \text{and} \quad \mathcal{A} - k := \{n \in \mathbb{N}\, : \, n+k \in \mathcal{A}\}.
\end{equation*}
\end{definition}
\endgroup

At the end of this section we show that the Diophantine density is not translation invariant, and we establish richness in Section \ref{sect:richness}. Out of the four properties listed above this leaves subadditivity and atomless-ness, which we were not able to establish.  But we present some evidence that it is not subadditive in Section \ref{sect:open_questions}.

In another study of generic densities, Leonetti and Tringali \cite{leonetti2020notions} define an upper density to be a function which satisfies Properties \ref{item:full_density}, \ref{item:monotonicity}, \ref{item:subadditivity} and \ref{item:translation}, plus an additional  property called $(-1)$\textit{-homogeneity}.  

\begingroup
\renewcommand{\thedefinition}{2b}
\begin{definition}  An abstract density may satisfy the following property:
    \begin{enumerate}[itemsep=0.25em, start=8, label=(\roman*)]
    \item $(-1)$\textit{-homogeneity:} $d(k \mathcal{A}) = k^{-1} d(\mathcal{A})$ for all $\mathcal{A} \subseteq \mathbb{N}$ and $k \in \mathbb{N}$. \label{item:dilation}
\end{enumerate}
\end{definition}
\endgroup
Our next result shows that the Diophantine density is not $(-1)$-homogeneous, and therefore cannot be an upper density in the sense of Leonetti and Tringali.
\begin{proposition}\label{prop:multiples_of_n}
    For all integers $m \geq 2$, the set of multiples of $m$ has full Diophantine density. That is, $\delta(m\mathbb{N}) = 1$.
\end{proposition}
\begin{proof}
    Let $(p,q) \in \mathbb{Z} \times \mathbb{N}$ be coprime with $q > m$, and let
    \begin{equation*}
        N = \left\lfloor{\frac{q-1}{m}}\right\rfloor.
    \end{equation*}
    Next, let $\xi = mp/q$ in Dirichlet's approximation theorem, so that we have
    \begin{equation*}
        \min_{1\leq n\leq N}\dnint{\frac{nmp}{q}} \leq \frac{1}{\lfloor{N}\rfloor + 1}.
    \end{equation*}
    Setting $v = nm$, we see that by construction, $v \in m\mathbb{N}$ and
    \begin{equation*}
         v \leq m \left\lfloor{\frac{q-1}{m}}\right\rfloor < q.
    \end{equation*}
    Moreover,
    \begin{equation*}
        \lfloor{N}\rfloor + 1 = \left\lfloor{\frac{q-1}{m}}\right\rfloor + 1 \geq \frac{q}{m},
    \end{equation*}
    which implies that
    \begin{equation*}
        \frac{1}{\lfloor{N}\rfloor + 1} \leq \frac{m}{q}.
    \end{equation*}
    Putting this all together we see that
        \begin{equation*}
        \min_{v \in m\mathbb{N}, v < q}\dnint{\frac{vp}{q}} \leq \frac{m}{q} \ll \frac{1}{q}.
    \end{equation*}
    This shows that $\delta(m\mathbb{N}) \geq 1$, and in light of the upper bound derived in the Proposition \ref{prop:delta_bounds}, we have $\delta(m\mathbb{N}) = 1$.
\end{proof}
Proposition \ref{prop:multiples_of_n} provides an example of an interesting phenomenon.  Namely, the Diophantine density admits sets that possess co-infinite subsets of the same density.  We revisit this example in Section \ref{sect:open_questions}, where a related open question is posed.

In the final proposition of this section we show that the Diophantine density is not translation invariant.

\begin{proposition}
    The Diophantine density is not translation invariant.
\end{proposition}
\begin{proof}
    Let $m \geq 2$ be a positive integer and consider the set $2m\mathbb{N}$.  By Proposition \ref{prop:multiples_of_n} we know that $\delta(2m\mathbb{N}) = 1$.  We will demonstrate that a right translation by $m$ results in $\delta(2m\mathbb{N} + m) = 0$. 

    To see this, define $p_{n} = n$ and $q_{n} = 2mn+1$, so that
    \begin{equation*}
        \frac{p_{n}}{q_{n}} = \frac{n}{2mn+1} = \frac{1}{2m} - \frac{1}{2mq_{n}}.
    \end{equation*}
    For any $v \in 2m\mathbb{N} + m$, $v$ is of the form $2mu+m$, and therefore
    \begin{equation*}
        \frac{vp_{n}}{q_{n}} = \frac{2mu+m}{2m} - \frac{v}{2mq_{n}} = u + \frac{1}{2} - \frac{v}{2mq_{n}}.
    \end{equation*}
    Thus, for any such $v < q_{n}$ the distance to the nearest integer is given by
    \begin{equation*}
        \dnint{\frac{vp_{n}}{q_{n}}} = \frac{1}{2} - \frac{v}{2mq_{n}} > \frac{1}{2} - \frac{1}{2m} \geq \frac{1}{4}.
    \end{equation*}
    As there are infinitely many such rationals $p_{n}/q_{n}$, this rules out any positive exponent, thereby establishing that $\delta(2m\mathbb{N} + m) = 0$.
\end{proof}

As far as the author knows, the concept of Diophantine density has not been directly studied before, but related concepts have appeared in the literature.  One related idea that has appeared in the literature is that of a Heilbronn set.  A set $\mathcal{A} \subseteq \mathbb{N}$ is said to be Heilbronn if, for every real number $\xi$ and $\varepsilon > 0$, there exists a $v \in \mathcal{A}$ such that $\dnint{v \xi} < \varepsilon$.

Recently, Iyer \cite{iyer2025rational} defined two characterizations of sets that are more closely related to our notion of Diophantine density. Firstly, a set $\mathcal{A} \subseteq \mathbb{N}$ is said to be an $\mathcal{H}$-set, if there exists a function $f(x)$ decreasing to $0$, such that for every $\xi \in \mathbb{R}$ and $N \in \mathbb{N}$, one has
\begin{equation*}
    \min_{v \in \mathcal{A}, v \leq N}\dnint{v \xi} \leq f(N).
\end{equation*}

Iyer also defines a set $\mathcal{A} \subseteq \mathbb{N}$ to be a \textit{strong approximating} set, if for all $\theta < 1$ and $\xi \in \mathbb{R}$ one has
\begin{equation*}
    \min_{v \in \mathcal{A}, v \leq N}\dnint{v \xi} \leq C_{\theta}N^{-\theta}.
\end{equation*}
where $C_{\theta}$ is a constant depending only on $\theta$.

If $\mathcal{A}$ is a strong approximating set, then by taking $\xi = p/q$ and $N = q-1$ in the definition, we see that 
\begin{equation*}
    \min_{v \in \mathcal{A}, v < q}\dnint{\frac{vp}{q}} \leq C_{\theta}(q-1)^{-\theta} \ll_{\theta} q^{-\theta}
\end{equation*}
for all $\theta < 1$ and rational $p/q$.  This implies that any strong approximating set has a Diophantine density of $1$.

\section{Bounding the Diophantine density of a set}\label{sect:delta_bounds}
The goal of this section to prove two theorems which provide an upper and lower bound for the Diophantine density of a given set and its complement, respectively.  We begin by stating the following definition for the counting function of a subset of the natural numbers.

\begin{definition}
    Let $\mathcal{A} \subseteq \mathbb{N}$ and for $n \in \mathbb{N}$ let
    \begin{equation*}
        \mathcal{A}(n) := \# \{a \in \mathcal{A} \, : \, a < n\} 
    \end{equation*}
    denote the number of elements of $\mathcal{A}$ less than $n$.
\end{definition}

\begin{remark}
    We would like to point out that our convention for the counting function $\mathcal{A}(n)$ differs from the usual definition, in which $\mathcal{A}(n)$ is defined as the number of elements less than \textit{or equal to} $n$.
\end{remark}

The first goal of this section is to establish a lower bound for the density of a complement $\mathbb{N} \,\backslash \, \mathcal{A}$.  To accomplish this, we need the following lemma.

\begin{lemma}\label{lem:complement_bound}
Let $\mathcal{A} \subseteq \mathbb{N}$. Then for any coprime pair $(p,q) \in \mathbb{Z} \times \mathbb{N}$, 
\begin{equation*}
\min_{v\in \mathbb{N} \, \backslash\, \mathcal{A}, v < q}\dnint{\frac{v p}{q}} \leq \frac{\mathcal{A}(q)+1}{q}.
\end{equation*}
\end{lemma}

\begin{proof}
Fix a coprime $(p,q) \in \mathbb{Z} \times \mathbb{N}$.  Our goal is to bound the smallest element left in the multiset
\begin{equation*}
    D(p,q):= \left\{\dnint{\frac{k p}{q}}\,:\,  1 \leq k < q \right\}
\end{equation*}
after removing those elements for which $k \in \mathcal{A}$.

We begin by noting that since $p$ permutes the nonzero residue classes modulo $q$ and
\begin{equation*}
    \dnint{\frac{kp}{q}} = \dnint{\frac{kp \mod q}{q}}
\end{equation*}
we see that 
\begin{equation*}
    D(p,q)= \left\{\dnint{\frac{k}{q}}\,:\,  1 \leq k < q \right\}
\end{equation*}
as multisets.  Moreover, for $k < q$ we have
\begin{equation*}
    \dnint{\frac{k}{q}} = \frac{1}{q}\min\{k,q-k\},
\end{equation*}
and thus
\begin{equation*}
    D(p,q)= \left\{\frac{1}{q},\frac{1}{q},\frac{2}{q},\frac{2}{q},\dots,\frac{\lfloor{q/2}\rfloor}{q}\right\},
\end{equation*}
where each distinct element occurs twice, except for the last which occurs only once for even $q$.

Now consider what happens when we remove the elements of $D(p,q)$ which correspond to the $\mathcal{A}(q)$ elements of $\mathcal{A}$ less than $q$.  Since each distinct element of $D(p,q)$ occurs at least once, the smallest remaining element may be roughly bounded by 
\begin{equation*}
\min_{v\in\mathbb{N} \,\backslash \,\mathcal{A}, v < q}\dnint{\frac{v p}{q}} \leq \frac{\mathcal{A}(q)+1}{q}.
\end{equation*}
With that our proof is complete.
\end{proof}

\begin{theorem}\label{thm:delta_lower_bound}
    Let $\mathcal{A} \subseteq \mathbb{N}$ be non-empty.  Then $\delta(\mathbb{N} \,\backslash \,\mathcal{A})$ is bounded below by
\begin{equation*}
    \delta(\mathbb{N} \,\backslash \,\mathcal{A}) \geq 1-\limsup_{q\to\infty}\frac{\log \mathcal{A}(q)}{\log q}
\end{equation*}
\end{theorem}
\begin{proof}
Let us begin by defining  
\begin{equation*}
    m(\mathcal{A}) := \inf\{ \tau : \mathcal{A}(q) \ll q^{\tau}\}.
\end{equation*}
Since $\mathcal{A}$ is non-empty we have $0 \leq m(\mathcal{A}) \leq 1$.  If $m(\mathcal{A}) = 1$, then the lower bound $ \delta(\mathbb{N} \,\backslash \,\mathcal{A}) \geq 0$ follows by Proposition \ref{prop:delta_bounds}.  Therefore, let us assume that $m(\mathcal{A}) < 1$.  By Lemma \ref{lem:complement_bound}, we see that for any coprime pair $(p,q) \in \mathbb{Z} \times \mathbb{N}$, and $\tau > m(\mathcal{A})$,
\begin{align*}
\min_{v\in\mathbb{N} \,\backslash \,\mathcal{A}, v < q}\dnint{\frac{v p}{q}} \leq \frac{\mathcal{A}(q)+1}{q} \ll \frac{q^{\tau}+1}{q} \ll q^{-1+\tau}.
\end{align*}
Consequently, $\delta(\mathbb{N} \,\backslash \,\mathcal{A}) \geq 1- \tau$ for any $\tau > m(\mathcal{A})$.  This implies that
\begin{equation*}
    \delta(\mathbb{N} \,\backslash \,\mathcal{A}) \geq 1- m(\mathcal{A}).
\end{equation*}
    Our goal now is to show that
\begin{equation*}
    m(\mathcal{A}) = \limsup_{q\to \infty}\frac{\log \mathcal{A}(q)}{\log q}.
\end{equation*}
To achieve this, let $L$ denote the limit supremum in the previous line.  Then we see that for large enough $q$,
\begin{equation*}
    \frac{\log \mathcal{A}(q)}{\log q} \leq L + \varepsilon
\end{equation*}
for every $\varepsilon > 0$.  Thus, $\mathcal{A}(q) \leq q^{L+\varepsilon}$ and consequently, $m(\mathcal{A}) \leq L+\varepsilon$.  Since this holds for every $\varepsilon$ we obtain $m(\mathcal{A}) \leq L$.

For the reverse inequality, suppose that $\mathcal{A}(q) \ll q^\tau$. Then there exists a constant $C > 0$ such that, for all sufficiently large $q$,
\begin{equation*}
    \frac{\log \mathcal{A}(q)}{\log q}
    \leq \tau + \frac{\log C}{\log q}.
\end{equation*}
Taking the limit supremum gives $L \leq \tau$.  Since this holds for every $\tau$ for which $\mathcal{A}(q) \ll q^\tau$, we obtain $L \leq m(\mathcal{A})$.

Combining the upper and lower bounds for $m(\mathcal{A})$, we obtain $m(\mathcal{A}) = L$, which completes the proof.
\end{proof}

\begin{theorem}\label{thm:delta_upper_bound}
    Let $\mathcal{A} \subseteq \mathbb{N}$ be non-empty. Then $\delta(\mathcal{A})$ is bounded above by
\begin{equation*}
    \delta(\mathcal{A}) \leq \liminf_{q \to \infty}\frac{\log \mathcal{A}(q)}{\log q}.
\end{equation*}
\end{theorem}
\begin{proof}
The first step in our proof is to show that for every $0 \leq \delta < \delta(\mathcal{A})$ there exists a $c > 0$ such that $\mathcal{A}(q) \geq cq^{\delta}$ for all $q$ sufficiently large.  To accomplish this task, we will essentially examine the set of numerator classes modulo $q$, which are covered by the elements of $\mathcal{A}$ less than $q$.  

For fixed $v$ and $\delta$, define the set of numerator classes modulo $q$ by
\begin{equation*}
    P_{\delta}(v,q) := \left\{p \in (\mathbb{Z}/q\mathbb{Z})^{\times} \, :\, \dnint{\frac{vp}{q}} \leq  Cq^{-\delta}\right\},
\end{equation*}
then the definition of Diophantine density gives
\begin{equation*}
    (\mathbb{Z}/q\mathbb{Z})^{\times} = \bigcup_{\substack{ v \in \mathcal{A} \\ v < q}}P_{\delta}(v,q).
\end{equation*}

Now let us suppose that $q$ is prime.  Because $v < q$ and consequently $\gcd(v,q) = 1$, the map $p \mapsto vp\,(\mathrm{mod}\, q)$ is a bijection.  Therefore, the number of elements in $P_{\delta}(v,q)$ is equal to the nonzero residues $r$ modulo $q$ which satisfy $\dnint{r/q} \leq  Cq^{-\delta}$.  In other words,
\begin{equation*}
    \left\vert P_{\delta}(v,q) \right\vert =  \#\left\{1\leq r \leq q-1 \, :\, \dnint{\frac{r}{q}} \leq  Cq^{-\delta}\right\}.
\end{equation*}

This fact allows us to easily bound the order of $P_{\delta}(v,q)$.  To this end, observe that if $r$ is such that $\dnint{r/q} \leq  Cq^{-\delta}$, then by definition
    \begin{equation*}
        \dnint{\frac{r}{q}} = \frac{1}{q}\min\{r,q-r\} \leq Cq^{-\delta}.
    \end{equation*}
    This implies that either 
    \begin{equation*}
        r \leq Cq^{1-\delta} \quad \text{or} \quad q-r \leq  Cq^{1-\delta}
    \end{equation*}
    and thus, $r$ must satisfy
    \begin{equation*}
        r \in [1, Cq^{1-\delta}] \cup [q - Cq^{1-\delta},q-1].
    \end{equation*}
    So for fixed $v$ there are at most $2 Cq^{1-\delta}$ such numerator classes that $v$ can cover.  That is, $\left\vert P_{\delta}(v,q)\right\vert \leq 2 Cq^{1-\delta}$.
    
    By taking the union of all $v \in \mathcal{A}$ such that $v < q$, we find that
    \begin{equation*}
    q-1 \leq \sum_{\substack{ v \in \mathcal{A} \\ v < q}}\vert P_{\delta}(v,q) \vert \leq 2 C \mathcal{A}(q) q^{1-\delta} .
\end{equation*}
    Consolidating constants, this implies that for some $c > 0$,
    \begin{equation*}
        \mathcal{A}(q) \geq  c q^{\delta},
    \end{equation*}
    for all sufficiently large prime $q$.

Next, by Bertrand's postulate, for every sufficiently large integer
$n$, there exists a prime $q$ such that $n/2 \leq q \leq n$.
Since the counting function is non-decreasing and $\delta \geq 0$,
\begin{equation*}
    \mathcal{A}(n) \geq \mathcal{A}(q)
    \geq c q^\delta
    \geq c 2^{-\delta} n^\delta.
\end{equation*}

Thus, after adjusting the constant, the bound
$\mathcal{A}(q) \geq c q^\delta$ holds for all sufficiently large
integers $q$, not only primes.
   
    Finally, by taking logarithms, we obtain
    \begin{equation*}
        \frac{\log \mathcal{A}(q)}{\log q} \geq \delta + \frac{\log c}{\log q}.
    \end{equation*}
    Therefore,
    \begin{equation*}
        \delta \leq \liminf_{q \to \infty}\frac{\log \mathcal{A}(q)}{\log q}. 
    \end{equation*}
    Since this holds for every $0 \leq \delta < \delta(\mathcal{A})$ we are able to say that
    \begin{equation*}
        \delta(\mathcal{A}) \leq \liminf_{q \to \infty}\frac{\log \mathcal{A}(q)}{\log q}, 
    \end{equation*}
    as we had sought to show.
\end{proof}

\begin{remark}
By combining Theorems \ref{thm:delta_lower_bound} and \ref{thm:delta_upper_bound} we obtain the bounds
\begin{equation*}
    1- \limsup_{q\to\infty}\frac{\log \mathcal{A}^{c}(q)}{\log q}\leq \delta(\mathcal{A}) \leq \liminf_{q \to \infty}\frac{\log \mathcal{A}(q)}{\log q},
\end{equation*}
where we have set $\mathcal{A}^{c} := \mathbb{N}\,\backslash\,\mathcal{A}$.  Unfortunately, this statement cannot produce a non-trivial value for both the lower and upper bound simultaneously.  That is, a lower bound that is greater than $0$ and an upper bound that is less than $1$.

    To see this, we first note that $\mathcal{A}^{c}(q) + \mathcal{A}(q) = q-1$.  Then suppose that we have a non-trivial lower bound, so that
    \begin{equation*}
        \limsup_{q\to\infty}\frac{\log \mathcal{A}^{c}(q)}{\log q} < 1.
    \end{equation*}
    It must be the case that $\mathcal{A}^{c}(q) = o(q)$ as $q\to\infty$, and therefore that $\mathcal{A}(q) = q-1-o(q)$.  Plugging this in to the upper bound we see that
        \begin{equation*}
        \liminf_{q \to \infty}\frac{\log \mathcal{A}(q)}{\log q} = 1.
    \end{equation*}
    Similarly, when the upper bound is less than $1$, the lower bound is seen to be $0$ by a similar argument.
\end{remark}

\begin{remark}
    The quantities
    \begin{equation*}
        \underline{\varepsilon}(\mathcal{A}) := \liminf_{q \to \infty}\frac{\log \mathcal{A}(q)}{\log q} \quad \text{and} \quad \overline{\varepsilon}(\mathcal{A}) := \limsup_{q \to \infty}\frac{\log \mathcal{A}(q)}{\log q}
    \end{equation*}
    are often called the \textit{lower} and \textit{upper logarithmic densities} of the set $\mathcal{A}$ (see for instance \cite{grekos2012gaps}).  Thus, the bound for $\delta(\mathcal{A})$ given in the previous remark may be written succinctly as
    \begin{equation*}
        1 - \overline{\varepsilon}(\mathcal{A}^{c}) \leq \delta(\mathcal{A}) \leq \underline{\varepsilon}(\mathcal{A}).
    \end{equation*}

    We point out in passing that upper logarithmic density is connected to the theory of Dirichlet series via the identity
    \begin{equation*}
        \overline{\varepsilon}(\mathcal{A}) = \inf\left\{\nu > 0 \, :\, \sum_{a \in \mathcal{A}}a^{-\nu} < \infty \right\}.
    \end{equation*}
    Here the quantity on the right is referred to as the \textit{abscissa of convergence}.  A proof of a more general statement can be found in Apostol's book \cite{apostol2012modular}.
\end{remark}

\section{Constructing a set with prescribed density}\label{sect:richness}
So far we have seen examples of sets which have a Diophantine density of 0 as well as sets which have a Diophantine density of 1.  In this section, we present a construction that shows that for any $x \in (0,1)$, there exists a set whose Diophantine density is $x$.  In other words, using the language of Section \ref{sect:diophantine_density}, we show that the Diophantine density is rich.
\begin{theorem}\label{thm:set_construction}
    For any $x \in (0,1)$, there exists a set $\mathcal{A}_{x}$ such that $\delta(\mathcal{A}_{x}) = x$.
\end{theorem}
\begin{proof}
    Define the integer sequences $(N_{k})_{k \geq 0}$ and $(d_{k})_{k \geq 0}$ as
    \begin{equation*}
        N_{k} := \lfloor{2^{xk}}\rfloor \quad \text{and} \quad d_{k} := \left\lfloor{\frac{2^{k}}{N_{k}}}\right\rfloor.
    \end{equation*}
    Then define the set
    \begin{equation*}
        \mathcal{A}_{x,k} := \{d_{k},2d_{k},\dots,N_{k}d_{k}\}
    \end{equation*}
    and the union
    \begin{equation*}
        \mathcal{A}_{x} := \bigcup_{k=0}^{\infty}\mathcal{A}_{x,k}.
    \end{equation*}
    Next, let us pick an arbitrary integer $p$ and positive integer $q$, such that $2^{k} < q \leq 2^{k+1}$.  By applying Dirichlet's approximation theorem to the rational number $d_{k}p/q$ we find that there exists some $1 \leq n \leq N_{k}$, such that
    \begin{equation*}
        \dnint{\frac{nd_{k}p}{q}} \leq \frac{1}{N_{k}+1}.
    \end{equation*}
    Thus, by setting $v := n d_{k} \in \mathcal{A}_{x,k}$ we find that $v \leq 2^{k} <  q$, and
    \begin{equation*}
        \dnint{\frac{vp}{q}} \leq \frac{1}{N_{k}+1} <  2^{-xk} \leq 2^{x}q^{-x}
    \end{equation*}
    since $q \leq 2^{k+1}$.  We have therefore shown that for every coprime pair $(p,q) \in \mathbb{Z} \times \mathcal{A}$, the bound
    \begin{equation*}
        \min_{v \in \mathcal{A}, v < q}\dnint{\frac{vp}{q}} \ll_{x} q^{-x},
    \end{equation*}
    holds. This demonstrates that the Diophantine density of $\mathcal{A}_{x}$ is at least $x$.

    To establish that the Diophantine density of $\mathcal{A}_{x}$ is not greater than $x$, we will start by showing that $\mathcal{A}_{x}(n) \ll n^{x}$.  
    
    Let $L := \lfloor{\log_{2}(n)}\rfloor$, so that $2^{L} \leq n \leq 2^{L+1}$.  Since overlaps can only reduce the size of the union we see that
\begin{equation*}
    \mathcal{A}_{x}(n) \leq \sum_{k=0}^{\infty}\mathcal{A}_{x,k}(n).
\end{equation*}
Moreover, we point out that
\begin{equation*}
    \mathcal{A}_{x,k}(n)  \leq \min\left\{N_{k},\left\lfloor{\frac{n}{d_{k}}}\right\rfloor\right\}.
\end{equation*}

Our plan is to split the sum into two parts, one over indices less than or equal to $L$ and one over those greater than $L$.  For $k \leq L$ we simply use the fact that $\vert{\mathcal{A}_{x,k}}\vert = N_{k} \leq 2^{xk}$ and $x > 0$ to establish the bound
\begin{equation*}
    \sum_{k=0}^{L}\mathcal{A}_{x,k}(n) \leq  \sum_{k=0}^{L}2^{xk} \leq \frac{2^{x}}{2^{x}-1}2^{xL} \leq \frac{2^{x}}{2^{x}-1}n^{x}.
\end{equation*}

For those indices greater than $L$ we take advantage of the spacing $d_{k}$ between elements of $\mathcal{A}_{x,k}$.  Since $N_{k} \leq 2^{xk}$, we have
\begin{equation*}
    \frac{2^{k}}{N_{k}} \geq 2^{(1-x)k}.
\end{equation*}
Then because $\floor{y} \geq y/2$ for any $y > 1$, we see that
\begin{equation*}
    d_{k} = \floor{\frac{2^{k}}{N_{k}}} \geq \frac{1}{2}2^{(1-x)k}.
\end{equation*}
Thus, we obtain
\begin{equation*}
    \mathcal{A}_{x,k}(n)  \leq \frac{n}{d_{k}} \leq 2n2^{-(1-x)k}.
\end{equation*}
Using the fact that $x < 1$, we find that
\begin{equation*}
    \sum_{k=L+1}^{\infty}\mathcal{A}_{x,k}(n) \leq 2n  \sum_{k=L+1}^{\infty}2^{-(1-x)k} = \frac{2n 2^{-(1-x)(L+1)}}{1- 2^{-(1-x)}}.
\end{equation*}
Since $n < 2^{L+1}$, we find that $2^{-(1-x)(L+1)} < n^{-(1-x)}$, which shows that 
\begin{equation*}
    \sum_{k=L+1}^{\infty}\mathcal{A}_{x,k}(n)  \leq \frac{2 n^{x}}{1- 2^{-(1-x)}}.
\end{equation*}
Combining the two halves of the sum gives the expected bound, namely
\begin{equation*}
    \mathcal{A}_{x}(n) \ll_{x} n^{x}.
\end{equation*}
Finally, with $C_{x}$ denoting the implied constant in the bound above, an application of Theorem \ref{thm:delta_upper_bound} shows that
\begin{equation*}
   \delta(\mathcal{A}_{x}) \leq \liminf_{q \to \infty}\frac{\log \mathcal{A}_{x}(q)}{\log q} \leq \liminf_{q \to \infty}\frac{\log C_{x} q^{x}}{\log q} = x.
\end{equation*}
By combining this with the lower bound derived above we obtain $\delta(\mathcal{A}_{x}) = x$, completing the proof.
\end{proof}

\begin{example}
    As an example, let $x = 1/2$.  Then $\delta(\mathcal{A}_{1/2}) = 1/2$ and 
    \begin{equation*}
        \mathcal{A}_{1/2} = \{1,2,4,6,8,11,12,16,18,22,23,24,30,32,33,40,44,45,46,48,55,56,64,\dots\}.
    \end{equation*}
\end{example}

\section{Diophantine density and restricted approximation}
In this section we turn to the original motivation for this paper, namely, restricted Diophantine approximation.  The next theorem gives us an idea of how well an arbitrary irrational number can be approximated by a subset of $\mathbb{N}$.
\begin{theorem}\label{thm:main}
    Let $\mathcal{A} \subseteq \mathbb{N}$ and suppose that $\delta(\mathcal{A}) > 0$.  For any irrational number $\xi$ and real number $\delta$ satisfying $0 < \delta < \delta(\mathcal{A})$, there are infinitely many pairs $(u,v) \in \mathbb{Z} \times \mathcal{A}$ such that
    \begin{equation*}
     \left\vert \xi -  \frac{u}{v} \right\vert < \frac{1}{v^{1+\delta}}
    \end{equation*}
\end{theorem}
\begin{proof}
Let $p_{n}/q_{n}$ be the $n$-th convergent of $\xi$.  By a standard result from the theory of continued fractions (see e.g. \cite{schmidt1980}) we know that 
\begin{equation*}
    \left\vert{ \xi - \frac{p_{n}}{q_{n}}}\right\vert < \frac{1}{q_{n}^{2}}.
\end{equation*}

Next, fix a real number $\delta'$, such that $\delta < \delta' < \delta(\mathcal{A})$.  Since $\delta' < \delta(\mathcal{A})$, there is a  $v_{n} \in \mathcal{A}$ and constant $C := C(\mathcal{A},\delta')$ such that for large enough $n$, we have
\begin{equation*}
    \left\vert{\frac{p_{n}}{q_{n}} - \frac{N_{n}}{v_{n}}}\right\vert < \frac{C}{q_{n}^{\delta'}v_{n}},
\end{equation*}
where $N_{n}$ denotes the nearest integer to $v_{n} p_{n}/q_{n}$.

Thus, by the triangle inequality and the fact that $v_{n} < q_{n}$, we find that
\begin{align*}
    \left\vert \xi - \frac{N_{n}}{v_{n}} \right\vert &\leq \left\vert{ \xi - \frac{p_{n}}{q_{n}}}\right\vert + \left\vert \frac{p_{n}}{q_{n}} -\frac{N_{n}}{v_{n}} \right\vert \\
    &\leq \frac{1}{q_{n}^{2}} + \frac{C}{q_{n}^{\delta'} v_{n}} \\
    &\leq \frac{1}{v_{n}^{2}} + \frac{C}{v_{n}^{1+\delta'}}.
\end{align*}
By Proposition \ref{prop:delta_bounds}, we know that $\delta(\mathcal{A}) \leq 1$, so that $ 1+ \delta' < 2$.  Thus, we obtain the bound 
\begin{equation*}
    \left\vert \xi - \frac{N_{n}}{v_{n}} \right\vert \leq \frac{C+1}{v_{n}^{1+\delta'}}.
\end{equation*}

All that we have left to show is that $v_{n} \to \infty$ as $n \to \infty$.  If this were not the case, and there were only finitely many distinct denominators, then by the pigeonhole principle there must be a single $v$ for which $v_{n} = v$,  for infinitely many $n$.  Denote the subsequence of $n$'s for which this is true by $(n_{k})_{k\geq 0}$.  Since $\delta' > 0$, we find that
\begin{equation*}
    \dnint{v\xi} = \lim_{k\to\infty}\dnint{\frac{v_{n_{k}}p_{n_{k}}}{q_{n_{k}}}} \leq \lim_{k\to\infty}Cq^{-\delta'}_{n_{k}} = 0
\end{equation*}
This is clearly a contradiction, as $\xi$ is irrational.

Therefore, we conclude that there are infinitely many distinct $v_{n} \in \mathcal{A}$, such that
\begin{equation*}
    \left\vert \xi -  \frac{N_{n}}{v_{n}} \right\vert \leq \frac{C+1}{v_{n}^{1+\delta'}}.
\end{equation*}

Finally, since $\delta' -\delta > 0$, we have $C+1 < v_{n}^{\delta'-\delta}$ for all sufficiently large $n$.  Consequently, we may write
\begin{equation*}
    \left\vert \xi -  \frac{N_{n}}{v_{n}} \right\vert < \frac{v_{n}^{\delta'-\delta}}{v_{n}^{1+\delta'}} = \frac{1}{v_{n}^{1+\delta}}.
\end{equation*}
With that our proof is complete.
\end{proof}

\begin{definition}
    The \textit{restricted irrationality exponent} $\mu_{\mathcal{A}}(\xi)$, of an irrational number $\xi$ with respect to a set $\mathcal{A} \subseteq \mathbb{N}$, is defined as the supremum of all real numbers $\mu$ such that the inequality
\begin{equation*}
    \left\vert{\xi - \frac{p}{q}}\right\vert < \frac{1}{q^{\mu}}
\end{equation*}
possesses infinitely many solutions $(p,q) \in \mathbb{Z} \times \mathcal{A}$.  When $\mathcal{A} = \mathbb{N}$, we write $\mu(\xi)$ for the ordinary irrationality exponent of $\xi$.
\end{definition}

Just as Dirichlet's theorem guarantees that $\mu(\xi) \geq 2$ for any irrational $\xi$, our Theorem \ref{thm:main} gives us the following lower bound for the restricted irrationality exponent.

\begin{corollary}
    Let $\mathcal{A} \subseteq \mathbb{N}$, and suppose that $\delta(\mathcal{A}) > 0$.  For any irrational number $\xi$, the restricted irrationality exponent satisfies 
    \begin{equation*}
        \mu_{\mathcal{A}}(\xi) \geq 1 + \delta(\mathcal{A})
    \end{equation*}
\end{corollary}
\begin{proof}
    Theorem \ref{thm:main} tells us that for any $0 < \delta < \delta(\mathcal{A})$ there are infinitely many pairs $(p,q) \in \mathbb{Z} \times \mathcal{A}$ such that
    \begin{equation*}
        \left\vert{\xi - \frac{p}{q}}\right\vert < \frac{1}{q^{1+\delta}}.
    \end{equation*}
    This shows that $\mu_{\mathcal{A}}(\xi) \geq 1+\delta$ for every $0 < \delta < \delta(\mathcal{A})$, and thus $\mu_{\mathcal{A}}(\xi) \geq 1 + \delta(\mathcal{A})$, as expected.
\end{proof}

\section{Approximations involving Piatetski-Shapiro sequences}\label{sect:piatetski-shapiro}

In this section, we investigate approximations whose denominators either are elements or avoid elements of Piatetski-Shapiro sequences, which are sequences of the form $(\lfloor{n^{\tau}}\rfloor)_{n\geq 1}$ where $\tau > 1$.  Often, $\tau$ is defined to be non-integral, but here we do not impose such a restriction.  These sequences are named after Piatetski-Shapiro \cite{piatetski-shapiro1953}, who showed that for $\tau \in (1,12/11)$, there are infinitely many primes of the form $\lfloor{n^{\tau}}\rfloor$.

To prepare for the remainder of this section, let us define the set 
\begin{equation*}
    \mathcal{S}_{\tau} := \{\lfloor{n^{\tau}}\rfloor\,\vert\, n \in \mathbb{N}\},
\end{equation*}
consisting of the elements of a given Piatetski-Shapiro sequence.  The following lemma provides an upper and lower bound for the counting function of $\mathcal{S}_{\tau}$.

\begin{lemma}\label{lemma:ps_seq}
    Let $\tau >1$. The counting function for a Piatetski-Shapiro set satisfies the bound
    \begin{equation*}
    \mathcal{S}_{\tau}(n) \asymp n^{1/\tau}.
\end{equation*}
\end{lemma}
\begin{proof}
Our goal is to show that
\begin{equation*}
    \mathcal{S}_{\tau}(n) = \begin{cases}
        \lfloor{n^{1/\tau}}\rfloor &\text{if } n^{1/\tau} \notin \mathbb{N}, \\
        n^{1/\tau} - 1 &\text{if } n^{1/\tau} \in \mathbb{N},
    \end{cases}
\end{equation*}
which then implies the desired inequality, $\mathcal{S}_{\tau}(n) \asymp n^{1/\tau}$.  To accomplish this we will count the number of positive integers $k$ such that $\lfloor{k^{\tau}}\rfloor < n$.

The above inequality implies that $k^{\tau} < n$, and equivalently $k < n^{1/\tau}$.  Note that if $n^{1/\tau} \in \mathbb{N}$, then there are $n^{1/\tau} - 1$ positive integers $k$ less than $n^{1/\tau}$.  On the other hand, if $n^{1/\tau} \notin \mathbb{N}$, then there are $\lfloor{n^{1/\tau}}\rfloor$ such positive integers. With the value of the counting function $\mathcal{S}_{\tau}(n)$ established, our proof is now complete.  
\end{proof}

\begin{proposition}\label{prop:s_upper_bound}
    The Diophantine density of the Piatetski-Shapiro set $\mathcal{S}_{\tau}$, is bounded above by $\delta(\mathcal{S}_{\tau}) \leq 1/\tau$.
\end{proposition}
\begin{proof}
    Using Lemma \ref{lemma:ps_seq} along with Theorem \ref{thm:delta_upper_bound} easily establishes that
    \begin{equation*}
        \delta(\mathcal{S}_{\tau}) \leq \liminf_{q\to\infty}\frac{\log \mathcal{S}_{\tau}(q)}{\log q} = \frac{1}{\tau}.
    \end{equation*}
\end{proof}

\begin{proposition}
    The Diophantine density of $\mathcal{S}_{2}$, the set of perfect squares, possesses the bounds
    \begin{equation*}
        \frac{2}{7} \leq \delta(\mathcal{S}_{2}) \leq \frac{1}{2}.
    \end{equation*}
\end{proposition}
\begin{proof}
    We will begin by establishing the lower bound.  By setting $\xi = p/q$ and
    \begin{equation*}
        N = \lfloor{\sqrt{q}}\rfloor-1
    \end{equation*}
    in Zaharescu's theorem, we find that
    \begin{equation*}
        \min_{1\leq n \leq \lfloor{\sqrt{q}}\rfloor-1}\dnint{\frac{n^{2} p}{q}} \ll (\lfloor{\sqrt{q}}\rfloor-1)^{-4/7+\varepsilon} \ll q^{-2/7+\varepsilon/2}.
    \end{equation*}
    Thus, for every coprime pair $(p,q) \in \mathbb{Z} \times \mathbb{N}$,
    \begin{equation*}
        \min_{v\in \mathcal{S}_{2}, v < q}\dnint{\frac{v p}{q}} \ll q^{-\delta},
    \end{equation*}
    for any $\delta < 2/7$.  Consequently, by taking the supremum over all such $\delta$ we find that $\delta(\mathcal{S}_{2}) \geq 2/7$.

    For the upper bound, the previous proposition gives us $\delta(\mathcal{S}_{2}) \leq 1/2$, and our proof is complete.
\end{proof}

\begin{remark}
    It is an open question as to whether or not the value $4/7$ in Zaharescu's Theorem \ref{thm:zaharescu} can be replaced by $1$.  If this turns out to be true, then we would obtain the exact value $\delta(\mathcal{S}_{2}) = 1/2$.  

    Moreover, it is natural to wonder if, more generally, the bound given by Proposition \ref{prop:s_upper_bound} is sharp.  That is, if $\delta(\mathcal{S}_{\tau}) = 1/\tau$ for all $\tau > 1$.
\end{remark}

Next, let us take a look at Diophantine approximation using the complement of a Piatetski-Shapiro set.

\begin{proposition}\label{prop:piatetski-shapiro-dirichlet}
Let $\xi$ be an irrational number and let $\tau > 1$ be a real number.  For any $\varepsilon > 0$, there are infinitely many pairs $(p,q) \in \mathbb{Z} \times (\mathbb{N}\,\backslash \, \mathcal{S}_{\tau})$, such that
     \begin{equation*}
    \left\vert{\xi- \frac{p}{q}}\right\vert < \frac{1}{q^{2-1/\tau-\varepsilon}}.
\end{equation*} 
\end{proposition}
\begin{proof}
    From Lemma \ref{lemma:ps_seq} we know that $\mathcal{S}_{\tau}(n) \ll n^{1/\tau}$.  Then for some constant $C > 0$, we have
    \begin{equation*}
        \limsup_{n \to \infty}\frac{\log\mathcal{S}_{\tau}(n)}{\log n} \leq \limsup_{n \to \infty}\frac{\log(C n^{1/\tau})}{\log n} = \frac{1}{\tau}.
    \end{equation*}
    We can then use Theorem \ref{thm:delta_lower_bound} to produce the bound 
    \begin{equation*}
        \delta(\mathbb{N}\,\backslash\,\mathcal{S}_{\tau}) \geq 1 - \limsup_{n \to \infty}\frac{\log\mathcal{S}_{\tau}(n)}{\log n} \geq 1 - \frac{1}{\tau}.
    \end{equation*}
    Since we have required that $\tau > 1$ so that $1-1/\tau > 0$, Theorem \ref{thm:main} tells us that for any $\delta$ satisfying $0 < \delta < 1-1/\tau$, we have infinitely many pairs $(p,q) \in \mathbb{Z} \times (\mathbb{N}\,\backslash \, \mathcal{S}_{\tau})$, such that
    \begin{equation*}
        \left\vert{\xi- \frac{p}{q}}\right\vert < \frac{1}{q^{1 + \delta}}.
    \end{equation*}
    As this is equivalent to the stated result, our proof is complete.
\end{proof}

\begin{example}
    As an example, setting $\tau = 2$ in Proposition \ref{prop:piatetski-shapiro-dirichlet} shows that for any irrational number $\xi$ and any $\varepsilon > 0$, there are infinitely many pairs $(p,q) \in \mathbb{Z} \times \mathbb{N}$ where $q$ is not a perfect square, such that
\begin{equation*}
    \left\vert{\xi- \frac{p}{q}}\right\vert < \frac{1}{q^{3/2-\varepsilon}}.
\end{equation*}
\end{example}

\section{Open questions}\label{sect:open_questions}
There are a number of fundamental questions involving Diophantine density that have yet to be answered and that could provide promising avenues for further research.  We present a few such problems below.
\begin{question}
    Is the Diophantine density atomless?
\end{question}

\begin{question}
    It is immediate that $\delta(\mathcal{A} \cup \mathcal{B}) \geq \max\{\delta(\mathcal{A}), \delta(\mathcal{B})\}$.  Are there sets $\mathcal{A}$ and $\mathcal{B}$ such that $\delta(\mathcal{A} \cup \mathcal{B}) > \max\{\delta(\mathcal{A}), \delta(\mathcal{B})\}$?
\end{question}

Let us sketch one possible approach to this problem.  The idea is to construct two sets, such that each set is missing a subset of $\mathbb{N}$ which other set possesses.  So that separately, the Diophantine density of each set is limited by the missing subsets, but together, their union produces all of $\mathbb{N}$.

More specifically, let $x < 1$ and define the sequence of sets $(E_{n})_{n \geq 0}$ by
\begin{equation*}
    E_{n} := \left\{v \in \mathbb{N} \, : \, \dnint{\frac{vp_{n}}{q_{n}}} < Cq_{n}^{-x}\right\},
\end{equation*}
where $(p_{n}/q_{n})_{n \geq 0}$ is an associated sequence of rational numbers.   Now suppose we can construct the sequence $(p_{n}/q_{n})_{n \geq 0}$ recursively, so that the pairwise intersection of the even-indexed and odd-indexed sets of the sequence $(E_{n})_{n \geq 0}$ is empty.  That is,
\begin{equation*}
    \left(\bigcup_{n = 0}^{\infty}E_{2n}\right) \cap \left(\bigcup_{n = 0}^{\infty}E_{2n+1}\right) = \varnothing.
\end{equation*}

Then if we define
\begin{equation*}
    \mathcal{A} := \mathbb{N}\,/\,\left(\bigcup_{n = 0}^{\infty}E_{2n}\right) \quad \text{and} \quad \mathcal{B} := \mathbb{N}\,/\,\left(\bigcup_{n = 0}^{\infty}E_{2n+1}\right),
\end{equation*}
by construction we have $\max\{\delta(\mathcal{A}),\delta(\mathcal{B})\} \leq x$, and $\delta(\mathcal{A} \cup \mathcal{B}) = \delta(\mathbb{N}) = 1$.  Since we have required that $x < 1$, this construction gives a specific example of sets satisfying
\begin{equation*}
    \delta(\mathcal{A} \cup \mathcal{B}) > \max\{\delta(\mathcal{A}),\delta(\mathcal{B})\} .
\end{equation*}

Moreover, this would show that the Diophantine density is not subadditive. This is because for sufficiently small $x$, $\delta(\mathcal{A}) + \delta(\mathcal{B}) < 1$, and therefore
\begin{equation*}
    1 = \delta(\mathcal{A} \cup \mathcal{B}) > \delta(\mathcal{A}) + \delta(\mathcal{B}).
\end{equation*}

If Question 1 can be answered in the affirmative, then we pose a related follow-up question.  To state it, we will need the following definition.
\begin{definition}
    An infinite set $\mathcal{A} \subseteq \mathbb{N}$ is called \textit{minimally dense} with respect to the Diophantine density, if for every partition
    \begin{equation*}
        \mathcal{A} = \mathcal{A}_{1} \sqcup \mathcal{A}_{2}
    \end{equation*}
    such that $\vert{\mathcal{A}_{1}}\vert = \vert{\mathcal{A}_{2}}\vert = \infty$, we have
    \begin{equation*}
        \delta(\mathcal{A}) > \max\{\delta(\mathcal{A}_{1}),\delta(\mathcal{A}_{2})\}.
    \end{equation*}
\end{definition}
In other words, a minimally dense set cannot be partitioned into two infinite subsets, where either subset is as dense as the whole.  This behavior is in contrast to the behavior of $\mathbb{N}$, for instance, which is not minimally dense.  By Proposition \ref{prop:multiples_of_n} we know that $\mathbb{N} = m\mathbb{N} \sqcup (\mathbb{N}\, \backslash \, m\mathbb{N})$ gives a partition which violates the conditions of minimal density, since $\delta(\mathbb{N}) = \delta(m\mathbb{N})$.

\begin{question}
    Do minimally dense sets exist?  
\end{question}

\begin{question}
    How dependent is the Diophantine density of $\mathcal{A}$ on the growth of its counting function?  By Theorem \ref{thm:delta_upper_bound} we know that if $\mathcal{A}(q) \asymp q^{\alpha}$, then $0 \leq \delta(\mathcal{A}) \leq \alpha$.  Can $\delta(\mathcal{A})$ take a value that is strictly less than $\alpha$?
\end{question}

\section{Acknowledgments}
The author would like to thank Jeff Ylvisaker and Jerry Federspiel for taking an interest in this work as well as for their questions and comments which helped shape the content of this paper.

\begin{declaration}
    This work has benefited significantly from the assistance of ChatGPT 5.5, 5.6 and 6.  Notably, ChatGPT was instrumental in developing the proofs of Theorems \ref{thm:delta_lower_bound}, \ref{thm:delta_upper_bound} and \ref{thm:set_construction}.
\end{declaration}

\bibliographystyle{plain} 
\bibliography{refs} 
\end{document}